\documentclass[11pt]{amsart}
\usepackage{graphicx}
\usepackage{color, fullpage, amsmath, amsthm, amssymb,verbatim,enumerate}

\newtheorem{theorem}{Theorem}[section]

\newtheorem{lemma}[theorem]{Lemma}

\newtheorem{proposition}[theorem]{Proposition}
\newtheorem{definition}[theorem]{Definition}
\numberwithin{equation}{section}

\theoremstyle{definition}

\theoremstyle{definition}

\theoremstyle{definition}

\newcommand{\Z}{\mathbb{Z}}
\newcommand{\R}{\mathbb{R}}
\newcommand{\N}{\mathbb{N}}
\newcommand{\C}{\mathbb{C}}

\renewcommand{\S}{\mathbb{S}}

\begin{document}

\date{\today}
\title{Line, Spiral, Dense in Space}

\author{Atakora Agoro}
\email{aagoro1@niu.edu}
\address{Department of Mathematical Sciences, Northern Illinois University, DeKalb, IL 60115-2888, USA}

\author{Alastair N. Fletcher}
\email{afletcher@niu.edu}
\address{Department of Mathematical Sciences, Northern Illinois University, DeKalb, IL 60115-2888, USA}

\begin{abstract}
Dobbs proved that the second iterate of almost every line in the complex plane under the exponential function is dense in the plane. In this paper, we prove an analogous result for the second iterate of the Zorich map in $\R^3$.
\end{abstract}

\maketitle

\section{Introduction}
\label{s:intro}

It is a standard homework exercise in complex analysis to show that the image of an oblique line under $\exp$ is a spiral. It is much less standard to show that the image of almost all oblique lines under $\exp \circ \exp$ is a dense curve in $\C$. This remarkable result is due to Dobbs \cite[Theorem 1]{Dob15}. Dobbs' method was generalized by Vaskouski, Prochorov, and Sheshko \cite{VPS19} to cover a larger class of periodic functions in $\C$.

The aim of this paper is to prove an analogous result for a quasiregular version of the exponential function in $\R^3$. Such maps are called Zorich maps, as they were first constructed by Zorich \cite{Zor67} and the dynamics of these maps have been well studied in the literature, see for example the work of Bergweiler \cite{Ber10} and Tsantaris \cite{Tsa23}.

We will fix a particular Zorich map $\mathcal{Z}:\R^3 \to \R^3 \setminus \{ 0 \}$ to work with in Section \ref{s:prelims}.

\begin{theorem}
    \label{thm:1}
For almost every line $L$ in $\R^3$, the curve $(\mathcal{Z} \circ \mathcal{Z}) (L)$ is dense in $\R^3$.
\end{theorem}

The overarching strategy for proving Theorem \ref{thm:1} is the same as Dobbs, but there are more technical difficulties to overcome in $\R^3$. Given a ball $U$, we will study its second pre-image $(\mathcal{Z}\circ\mathcal{Z})^{-1}(U)$. Then instead of looking at a line to see whether it intersects this set, we will fatten it slightly, in a suitable sense, to give a cone and show that a certain proportion of this cone must intersect the second pre-image of the ball. We will then use measure theoretical considerations to obtain the result.

For the exponential map, the real axis is forward invariant, which means there are certainly some lines which must be excluded. For our Zorich map $\mathcal{Z}$, there are analogous lines that will need to be excluded, so we certainly cannot obtain Theorem \ref{thm:1} for every line in $\R^3$.

We leave open several related questions. First, Dobbs \cite[Theorem 3]{Dob15} showed that the image of every oblique line in $\C$ was dense in $\C$ under the third iterate of $\exp$. It seems plausible that the same might hold for $\mathcal{Z}$ in $\R^3$. Moreover, there is a whole family of Zorich maps that one could study in the context of generalizing Theorem \ref{thm:1} and, more generally, the same result could plausibly hold for any strongly automorphic map in $\R^3$ or higher dimensions. We, however, will satisfy ourselves with just proving Theorem \ref{thm:1} here.

The structure of the paper is as follows. In Section \ref{s:prelims} we will define the particular Zorich map we will use and recall various properties of it. In Section \ref{s:dense}, we will prove Theorem \ref{thm:1}.

This paper arises from the Ph.D. dissertation of the first named author. The external examiner for the dissertation was Neil Dobbs, and both authors would like to thank him for many invaluable comments and suggestions.

\section{Preliminaries}
\label{s:prelims}

\subsection{A Zorich map}
\label{ss:zorich}

Let $B$ denote the square 
\begin{equation}
\label{eq:B}    
B = \left \{ (x_1,x_2)\in \R^2 : |x_1|,|x_2| \leq \frac{\pi}{2} \right \} ,
\end{equation}
and let $\S^2_+$ denote the upper hemisphere in $\R^3$ given by
\[ \S^2_+ = \left \{ (x_1,x_2,x_3) \in \R^3 : x_1^2+x_2^2+x_3^2=1, x_3\geq 0 \right \}.\]
To build the Zorich map we will use, we start by defining the bi-Lipschitz map $h:B \to \S^2_+$ given by
\begin{equation}
    \label{eq:h}
h(x_1,x_2) = \left ( \frac{x_1 \sin M(x_1,x_2)}{\sqrt{x_1^2+x_2^2}} , \frac{x_2 \sin M(x_1,x_2)}{\sqrt{x_1^2+x_2^2}} , \cos M(x_1,x_2) \right ),
\end{equation} 
where $M(x_1,x_2) = \max \{ |x_1|,|x_2| \}$. By reflecting in the sides of squares in the domain and in the equator of $\S^2$ in the range, we obtain a doubly periodic map that, by an abuse of notation, we still call $h$ from $\R^2$ onto $\S^2$.

The Zorich map $\mathcal{Z}$ is then defined via
\begin{equation} 
\label{eq:Z}
\mathcal{Z} (x_1,x_2,x_3) = e^{x_3} h(x_1,x_2).
\end{equation}

If we let $B_0$ be the closed beam in $\R^3$ defined by $B_0 = B \times \R$, then it is clear that $\mathcal{Z}$ maps $B_0$ onto the closed upper half-space in $\R^3$ with the origin removed. Moreover, if we set 
\[ B_1 = \left [ \frac{\pi}{2} , \frac{3\pi}{2} \right ] \times \left [ -\frac{\pi}{2}, \frac{\pi}{2} \right ] \times \R \]
to be one of the beams neighbouring $B_0$, then the image of $B_1$ under $\mathcal{Z}$ is the closed lower half-space in $\R^3$ with the origin removed.

This produces a quasiregular map $\mathcal{Z} :\R^3 \to \R^3 \setminus \{ 0 \}$. At this stage, we point out that we do not really need that $\mathcal{Z}$ is quasiregular. It is only important for us that $\mathcal{Z}$ is an analogue of $\exp$ in $\R^3$. In fact, restricting to either the $\{x_1 = 0\}$ plane of the $\{x_2 = 0\}$ plane yields embedded copies of $\exp$. We refer to, for example, \cite[Appendix A2]{FlePra21} for a proof of the fact that $h$ defined in \eqref{eq:h} is bi-Lipschitz and $\mathcal{Z}$ is quasiregular. 

The branch set $\mathcal{B}(\mathcal{Z})$ is the set of points where $\mathcal{Z}$ fails to be a local homeomorphism. From the construction, we see that the branch set consists of the lines at the edges of faces of the beams, that is
\[ \mathcal{B}(\mathcal{Z}) = \bigcup_{ j,k \in \Z} \left ( \left \{ \frac{\pi}{2} + j\pi  \right \} \times \left \{ \frac{\pi}{2} + k\pi \right \} \times \R \right ) .\]

\subsection{The group of automorphisms}

It turns out that Zorich maps are strongly automorphic with respect to a discrete group of isometries. This means that there is such a group $G$ with 
\begin{enumerate}[(a)]
\item $\mathcal{Z}(g(x)) = \mathcal{Z}(x)$ for any $x\in \R^3$ and any $g\in G$, and
\item if $\mathcal{Z}(x) = \mathcal{Z}(y)$, then there exists $g\in G$ with $y=g(x)$.
\end{enumerate}

It is clear from the construction that $\mathcal{Z}$ is periodic with periods $(2\pi,0,0)$ and $(0,2\pi,0)$. However, $G$ also contains rotations by $\pi$ about lines in the branch set. If we set $g_1(x) = x + (2\pi,0,0)$, $g_2(x) = x + (0,2\pi , 0)$ and 
\[ g_3(x_1,x_2,x_3) = (\pi-x_1,\pi - x_2 , x_3), \]
then $G$ is generated by $g_1,g_2$, and $g_3$.

As $\mathcal{Z}$ maps $\operatorname{int}(B_0)$ bijectively onto the open upper half-space, and maps $\operatorname{int}(B_1)$ bijectively onto the open lower half-space, it follows that $\operatorname{int}(B_0 \cup B_1)$ together with some of its boundary forms a fundamental domain for the action of $G$ that we call $\Omega$. We will see later that this {\it domino} behaviour where a fundamental domain is naturally split into two regions has large consequences for our analysis.

\subsection{Mapping properties of our Zorich map}

Here we record some important mapping properties of $\mathcal{Z}$ that we will need later.

\begin{lemma}
    \label{lem:1}
For $t\in \R$, let $H_t$ be the plane 
\[ H_t = \{ (x_1,x_2,t) \in \R^3 : x_1,x_2\in \R \} .\]
Then $\mathcal{Z}$ maps $H_t$ onto the sphere $S(0,e^t) = \{ x\in \R^3 : |x| = e^t \}$. Conversely, $\mathcal{Z}^{-1}(S(0,r)) = H_{\ln r}$ for any $r>0$.
\end{lemma}

This follows directly from the construction of $\mathcal{Z}$ and so we omit the proof.

\begin{lemma}
    \label{lem:2}
Let $t>0$ and let $S_t \subset B_0$ be defined via
\[ S_t = \{(x_1,x_2,x_3) \in B_0 : e^{x_3} \cos M(x_1,x_2) = t \}.\]
Then $\mathcal{Z}$ maps $S_t$ bijectively onto the plane $H_t$.
\end{lemma}

\begin{proof}
As $\mathcal{Z}$ is injective on $B_0$, it is enough to find the pre-image of $H_t$ under $\mathcal{Z}$ contained in $B_0$. By \eqref{eq:h} and \eqref{eq:Z}, this follows by solving $e^{x_3} \cos M(x_1,x_2) = t$. Re-writing this expression, we also obtain that
\[ S_t = \left\{(x_1,x_2,x_3) \in B_0 : x_3 = \ln \left ( \frac{t}{\cos M(x_1,x_2) } \right ) \right \}.\]
\end{proof}

For $t<0$, we define $S_t\subset B_1$ to be the reflection of $S_{-t}$ across the common face of $B_0$ and $B_1$.

\begin{lemma}
    \label{lem:3}
Let $t\in \R$. Then:
\begin{enumerate}
\item for $t>0$, $\mathcal{Z}^{-1}(H_t)$ is $\bigcup_{g\in G} g(S_t)$,
\item for $t<0$, $\mathcal{Z}^{-1}(H_t)$ is $\bigcup_{g\in G} g(S_t)$,
\item for $t=0$, $\mathcal{Z}^{-1}(H_t)$ is the union of all of the faces of $g(B_0)$ and $g(B_1)$ over all $g\in G$.    
\end{enumerate}
\end{lemma}

\begin{proof}
Part~(a) follows from Lemma \ref{lem:2} and using the strong automorphy of $\mathcal{Z}$ with respect to $G$. Part~(b) follows analogously, as if $t<0$, then by construction, $\mathcal{Z}$ maps $S_t$ bijectively onto $H_t$. Part~(c) follows immediately by the construction of $\mathcal{Z}$. 
\end{proof}

It is clear that $S_t$ is a cone with vertex $(0,0,\ln t)$ and that every cross-section of $S_t$ with a plane $\{x_3 = s\}$ for $s > \ln t$ is a square given by $M(x_1,x_2) = \cos^{-1}[te^{-s}]$. In particular, as $s\to \infty$, the cross-sections get closer and closer to the sides of the square beam that $S_t$ lies in. We make this observation precise with the following lemma.

\begin{lemma}
\label{lem:squarebeam}
Let $t> 0$. Given $\xi >0$, there exists $s_0 > 0$ such that if $x\in S_t$ with $x_3 > s_0$, then $\operatorname{dist}(x,\partial B_0) < \xi$.
\end{lemma}

By applying a translation or rotation, exactly the same result holds for $g(S_t)$ for any $t\neq 0$ and $g\in G$.

\begin{proof}
By the proof of Lemma \ref{lem:2}, $S_t$ is given by the equation
\[ x_3 = \ln \left ( \frac{ t}{\cos M(x_1,x_2) } \right ).\]
Suppose that $x = (x_1,x_2,x_3) \in S_t$ and set $y = \operatorname{dist}(x,\partial B_0)$. By construction, it follows that
\[ \cos (\pi/2 - y) = t e^{-x_3},\]
and thus
\[ \sin y = te^{-x_3}.\]
The lemma then follows from the continuity of the sine function.
\end{proof}

\subsection{A notion of distortion}

A tool used by Dobbs \cite{Dob15}, as well as in \cite{VPS19}, is the notion of the distortion of a differentiable, non-constant function $f$ on a set $E\subset \C$ via
\[ D(f,E) = \frac{ \sup_{z\in E} |f'(z)| }{\inf_{z\in E} |f'(z)|} .\]

As our Zorich map is not differentiable everywhere, we will need to modify this quantity in order to utilize it. In fact, as quasiregular maps are in general just locally H\"older, we will need to verify $\mathcal{Z}$ has better regularity than this. We start with a definition that comes up when dealing with bounded length distortion maps, or BLD maps for short.

\begin{definition}
    \label{def:infbilip}
Let $E \subset \R^n$ be a domain. A function $f:E \to \R^n$ is called infinitesimally bi-Lipschitz if there exists a constant $\lambda \geq 1$ such that for all $x_0 \in E$, we have
\[ \frac{1}{\lambda} \leq \liminf_{x\to x_0} \frac{ |f(x) - f(x_0) |}{|x-x_0|} \leq \limsup_{x\to x_0} \frac{ |f(x) - f(x_0)| }{|x-x_0|} \leq \lambda .\]
\end{definition}

For example, the extended map $h:\R^2 \to \S^2$ certainly is not bi-Lipschitz as it is doubly periodic and thus not injective. However, it is infinitesimally bi-Lipschitz with the same constant $\lambda$. Again we refer to \cite[Appendix A2]{FlePra21} for this fact.

Next, if $E\subset \R^n$ is a domain and $f:E \to \R^n$ is a non-constant continuous map, for $x\in E$ and $0<r<\operatorname{dist}(x,\partial E)$ we recall the linear dilatations
\[ L(x,f,r) = \sup_{|y-x| = r} |f(y) - f(x)|, \quad \ell(x,f,r) = \inf_{|y-x|=r} |f(y)-f(x)|.\]

\begin{definition}
    \label{def:df}
With the notation above, the upper and lower pointwise Lipschitz constants of $f$ at $x\in E$ are, respectively,
\[ \widetilde{L}_f (x) = \limsup_{r\to 0} \frac{ L(x,f,r)}{r}, \quad \widetilde{\ell}_f(x) = \liminf_{r\to 0} \frac{ \ell(x,f,r)}{r} .\]
The {\it relative distortion} of $f$ in $E$ is then defined by
\[ D(f,E) = \frac{ \sup_{x\in E} \widetilde{L}_f(x) }{ \inf_{x\in E} \widetilde{\ell}_f(x) },\]
whenever both quantities are finite and the denominator is strictly positive.
\end{definition}

Issues arise for $D(f,E)$ when $f$ is only a H\"older map. Consider for example the radial map $f(x) = x|x|^{a-1}$, for $0<a<1$. However, it is clear that if $f$ is infinitesimally bi-Lipschitz with constant $\lambda >1$, then $D(f,E) \leq \lambda^2$.

Recall that $h:\R^2 \to \S^2$ is infinitesimally bi-Lipschitz with constant, say, $\lambda(h)$. We will want to control the relative distortion of our Zorich map $\mathcal{Z}$ in certain regions, as given in the following lemma.

\begin{lemma}
\label{lem:4}
Let $t_1<t_2$ and let $E$ be the slab $\R^2 \times (t_1,t_2) \subset \R^3$. Then 
\[ D(\mathcal{Z} , E) \leq \lambda(h)^2 e^{t_2-t_1}.\]
\end{lemma}

\begin{proof}
It follows directly from \eqref{eq:Z} that if $x\in E$, then
\[ \frac{e^{t_1}}{\lambda} \leq \widetilde{\ell}_{\mathcal{Z}} ( x_1,x_2,x_3) \leq \widetilde{L}_{\mathcal{Z}}(x_1,x_2,x_3) \leq \lambda e^{t_2}.\]
The lemma follows.
\end{proof}

We will also need to record how relative areas change under a map with controlled relative distortion.

\begin{proposition}
    \label{prop:reldist}
Suppose that $E\subset \R^n$ is a Lebesgue measurable set with $m(E) > 0$ and $U \subset E$ is a measurable subset. If $f:E \to \R^n$ is a homeomorphism that satisfies $D(f,E) \leq \lambda$, then 
\[ \frac{1}{\lambda^n} \cdot \frac{m(U)}{m(E)} \leq \frac{ m(f(U))}{m(f(E))} \leq \lambda^n \cdot \frac{ m(U)}{m(E)}.  \]
\end{proposition}

\begin{proof}
By the hypotheses, the
Jacobian \(J_f\) exists almost everywhere on \(U\).
Moreover, we have, for any $x\in U$ where $J_f(x)$ exists,
\[ J_f(x) \leq \widetilde{L}_f (x) ^n, \quad J_f(x) \geq \widetilde{l}_f(x)^n.\]
It follows that if $f$ has relative distortion $D(f,U) = \alpha$, then
\[
\frac{\sup_{x\in U} J_f(x)}{\inf_{x\in U} J_f(x)} \le \frac{\sup_{x\in U} \widetilde{L}_f(x)^n}{\inf_{x\in U} \widetilde{l}_f(x)^n} \leq \alpha^n .
\]

Using the change-of-variables formula, we obtain
\[
m\bigl(f(E)\bigr)=\int_E J_f(x)\,dx
\;\ge\;
\inf_{x\in U} J_f(x)\, m(E),
\]
and
\[
m\bigl(f(U)\bigr)=\int_U J_f(x)\,dx
\;\le\;
\sup_{x\in U} J_f(x)\, m(U).
\]

Combining these estimates yields
\[
\frac{m\bigl(f(E)\bigr)}{m(f(U))}
\;\ge\;
\frac{\inf_{x\in U} J_f(x)}{\sup_{x\in U} J_f(x)}\,
\frac{m(E)}{m(U)}
\;\ge\;
\frac{1}{\alpha^n}\,
\frac{m(E)}{m(U)}.
\]
The other inequality follows analogously and so this completes the proof.
\end{proof}

We will apply a variation of this result when $U\subset E$ are measurable subsets of a surface in $\R^3$ - the conclusion holds with $n=2$.

\subsection{The Lebesgue Density Theorem}

A key ingredient to Dobbs' proof, and thus to ours, is the following measure theoretical result. Recall that if $E \subset \R^n$ is a Lebesgue measurable set and if $\epsilon >0$, then a point $x\in \R^n$ is called an $\epsilon$-density point of $E$ if 
\[ \lim_{r\to 0} \frac{ m(E \cap B(x,r))}{m(B(x,r))} \geq \epsilon , \]
if the limit exists. 

\begin{theorem}
    \label{thm:ldp}
Let $E\subset \R^n$ be a Lebesgue measurable set. Then almost every point of $E$ is a $1$-density point of $E$.
\end{theorem}

We refer to, for example, \cite[p.45]{EvaGar92} for this result.

\section{Dense curves}
\label{s:dense}

\subsection{The set-up}

In this section we will prove Theorem \ref{thm:1}. With the Zorich map $\mathcal{Z}$ as defined in Section \ref{ss:zorich}, we set $f = \mathcal{Z} \circ \mathcal{Z}$. Next, instead of considering all lines in $\R^3$, we will consider all lines passing through an arbitrary point $P=(p_1,p_2,p_3) \in \R^3$. 

Now, as any line passing through $P$ in the $x_3 = p_3$ plane will be mapped by $\mathcal{Z}$ into a sphere, and thus $f$ will map this line into a bounded region, we will remove these lines from consideration without affecting the result. Moreover, by the symmetry of the Zorich map, it is possible that lines in either plane $x_1 = p_1$ or $x_2 = p_2$ will be mapped by $\mathcal{Z}$ into the same plane (c.f. the comment that $\mathcal{Z}$ restricted to the $x_1= 0$ or $x_2=0$ planes is an embedded copy of $\exp$). Therefore the image of such a line under $f$ may be contained in a plane and thus cannot be dense in $\R^3$. For technical reasons, we will also remove the planes given by $x_2 = \pm x_1$. Removing lines in these planes from consideration will not affect the result.

Every other line that passes through $P$ may be parametrized as passing through a unique point of 
\begin{equation}
    \label{eq:para}
Y = \{ P+(x_1,x_2,x_3) : M(x_1,x_2)=1, x_3>0, x_1x_2(x_1^2-x_2^2)\neq 0 \}. 
\end{equation}
Then $Y$ is part of the boundary of a semi-infinite square beam. For $\alpha \in Y$, we denote by $L_{\alpha}$ the unique line through $P$ and $\alpha$.

Now, let $(q_n)_{n=1}^{\infty}$ be a dense sequence in $\R^3$ such that $|q_n|$ is not $0$ or $1$ for all $n$. Then let $(\delta_n)_{n=1}^{\infty}$ be a decreasing sequence with $\delta_n \to 0$ and $0<\delta_n < |q_n|$ so that 
\[ \mathcal{U} = \{ B(q_n,\delta_n) : n\geq 1 \} \]
forms a countable base for the usual topology of $\R^3 \setminus (S^{2} \cup \{ 0 \})$. The fact we have excluded $0$ and the unit sphere here will not affect the result, as if a set is dense in $\R^3 \setminus (S^{2} \cup \{ 0 \})$, then it is dense in $\R^3$.

Observe that $f(L_{\alpha})$ is dense in $\R^3 \setminus (S^{2} \cup \{ 0 \})$ if and only if $f(L_{\alpha}) \cap U \neq \emptyset$ for all $U\in \mathcal{U}$. Let us therefore define $X_U$ to be
\[ X_U = \left \{ \alpha \in Y : f(L_{\alpha}) \cap U \neq \emptyset \right \}.\]
As each $U \in \mathcal{U}$ is open and $f$ is continuous, it follows that $X_U$ is also open, and hence Lebesgue measurable. If we set 
\[ X' = \bigcap_{U \in \mathcal{U}} X_U,\]
then if each $X_U$ has full measure in $\R^2$, so does $X'$, as a countable intersection of full measure sets. Now, given a set $X\subset \R^2$, if almost every point in $\R^2$ is an $\epsilon$-density point of $X$ for some $\epsilon >0$, then the set of $1$-density points for the complement of $X$ has zero measure, and hence by the Lebesgue Density Theorem, Theorem \ref{thm:ldp}, $X$ has full measure. Proving Theorem \ref{thm:1} then reduces to proving the following theorem.

\begin{theorem}
    \label{thm:2}
Let $U\subset \mathcal{U}$. Then there exists $\epsilon >0$ such that each $\alpha \in Y$ is an $\epsilon$-density point for $X_U$.
\end{theorem}

\subsection{Pre-images under the Zorich map}

Towards a proof of Theorem \ref{thm:2}, we suppose that $U \in \mathcal{U}$ and we set $x_0 = q_n$ and $r_0 = \delta_n$ so that $U = B(x_0,r_0)$. Recall that $0<r_0<|x_0|$ so that $\overline{U}$ avoids the origin. Moreover, $|x_0|\neq 1$. We set $R = \partial B(0,|x_0|)$, $H = \mathcal{Z}^{-1}(R)$, $\widetilde{U} = U\cap R$, $V = \mathcal{Z}^{-1}(U)$, and $\widetilde{V} =\mathcal{Z}^{-1}(\widetilde{U})$. By Lemma \ref{lem:1}, we have $H = \{ x_3 = \ln |x_0| \}$. As $|x_0| \neq 1$, $H$ is not the $x_3 = 0$ plane.

Recall that $h$ from \eqref{eq:h} is $\lambda$-bi-Lipschitz for some $\lambda >1$, and then that if $h$ is extended by reflections to all of $\R^2$, then $h$ is infinitesimally $\lambda$-bi-Lipschitz. We use this to show that there are disks of definite radius in each component of $\widetilde{V} \subset H$.

\begin{lemma}
    \label{lem:radius}
Any connected component of $\widetilde{V} \subset H$ contains a disk of radius at least $r_0(8\lambda e^{|x_0|})^{-1}$.
\end{lemma}

\begin{proof}
We split into three cases, depending on the location of $\widetilde{U}$ in $U$ relative to the $x_3=0$ plane.
\begin{enumerate}[(a)]
\item For the first case, $\widetilde{U} \subset R \cap \{ x_3 > 0 \}$. Then $\widetilde{U}$ is a spherical disk sitting on $R$ of diameter at most $2r_0$. There is a branch of $\mathcal{Z}^{-1}$ that maps $\widetilde{U}$ homeomorphically and, in fact, in a $\lambda e^{|x_0|}$-bi-Lipschitz way, onto a topological disk contained in $\operatorname{int}(B_0) \cap H$. It follows that this component of $\widetilde{V}$ contains a disk of radius at least $r_0(\lambda e^{|x_0|})^{-1}$. All other components of $\widetilde{V}$ are obtained from this one under the action of $G$. As $G$ consists of isometries, the same lower bound holds for all components of $\widetilde{V}$.
\item For the second case, $\widetilde{U} \subset R \cap \{x_3 < 0\}$. This is analogous to the first case, except this time the initial component of $\widetilde{V}$ is contained in $\operatorname{int}(B_1) \cap H$.
\item For the final case, we consider when $\widetilde{U}$ intersects the equator $R \cap \{x_3 = 0\}$. Then either $\widetilde{U} \cap \{x_3>0\}$ or $\widetilde{U} \cap \{x_3<0\}$ contains a spherical disk of diameter $r_0/4$. By applying the argument in the previous two cases to whichever situation we are in, we obtain a disk of radius at least $r_0(8\lambda e^{|x_0|})^{-1}$ in each component of $\widetilde{V}$.
\end{enumerate}
\end{proof}

This lemma shows that each component of $\mathcal{Z}^{-1}(U)$ contains disks of a definite size contained in the plane $H$. We need to take another pre-image under $\mathcal{Z}$ and for our purposes, it will suffice to consider pre-images of $H$. However, this has already been addressed by Lemma \ref{lem:2} and Lemma \ref{lem:3}. As $|x_0| \neq 1$, we have $\ln |x_0| \neq 0$ and so for convenience, we will set $S = S_{\ln |x_0|}$. We recall that this means that if $|x_0| >1$ then $S \subset B_0$ and if $0<|x_0|<1$ then $S\subset B_1$. We will also set $S_g = g(S)$ and so by Lemma \ref{lem:3}, we have $f^{-1}(R) = (\mathcal{Z} \circ \mathcal{Z})^{-1}(R) = \bigcup_{g\in G} S_g$.

Each $S_g$ is a cone with square cross-sections and a vertex lying in the plane $\{x_3 = \ln \left |\ln |x_0| \right |\}$ and, moreover, $S_g$ has four faces. Pick one of the faces and call it $F$. For $\ln \left |\ln |x_0| \right | < t_1 <t_2$, denote by $F(t_1,t_2)$ the region
\begin{equation} 
\label{eq:Ftt}
F(t_1,t_2) := F \cap \{ t_1 < x_3 < t_2 \} .
\end{equation}
Informally, $F(t_1,t_2)$ is almost a flat rectangle.

\begin{lemma}
    \label{lem:estimate}
Given $|x_0| \in \R^+ \setminus  \{ 0 ,1 \}$, there exists $a>0$ such that if $t_1 > \ln \left |\ln |x_0| \right |$ and $t_2 > t_1+a$, then 
\[ \frac{e^{t_2}}{\sqrt{2}} - e^{t_1} > 2\pi .\]
Moreover, we may choose $a$ so that $e^{2a}>3$.
\end{lemma}

\begin{proof}
Via a direct computation, we see that we can take
\[ a > \ln \left [ \sqrt{2} \left ( \frac{2\pi}{|\ln|x_0||} + 1 \right ) \right ].\]
For the final claim, we just increase $a$ if necessary.
\end{proof}

\begin{lemma}
\label{lem:doublepreims}
Let $t_1 > \ln \left |\ln |x_0| \right |$ and let $t_2 > t_1+a $, where $a$ is from Lemma \ref{lem:estimate}. Then 
\[ \frac{ m(f^{-1}(U) \cap F(t_1,t_2) ) }{m(F(t_1,t_2))} \geq C,\]
where $C$ depends only on the input data $r_0,|x_0|$ and $\lambda$.
\end{lemma}

\begin{proof}
First, by construction $\mathcal{Z}$ maps $F$ onto a quadrant of the plane $H$. Without loss of generality, we may assume this quadrant is
\[ Q = \{ (x_1,x_2,x_3) : x_1 \geq |x_2|, x_3 = \ln |x_0| \}.\]
Our analysis for the other quadrants runs analogously. Then $\mathcal{Z}$ maps $F(t_1,t_2)$ onto a logarithmic square in $Q$. That is, two sides of $\mathcal{Z} ( F(t_1,t_2))$ run along the lines $x_1 = \pm x_2$, and the other two sides are arcs of circles with radii $e^{t_1}$ and $e^{t_2}$ respectively.

In particular, $\mathcal{Z} ( F(t_1,t_2))$ contains a trapezoid $T$ in $Q$ whose vertical line segments are contained in the lines $x_1 = e^{t_1}$ and $x_1 = e^{t_2}/\sqrt{2}$, and the other boundary line segments are contained in the lines $x_1 = \pm x_2$. By the hypotheses and Lemma \ref{lem:estimate}, it follows that $T$ is non-trivial. 

If we compare the areas of $\mathcal{Z} ( F(t_1,t_2))$ and $T$, we have
\[ m(T) = (e^{t_2}/\sqrt{2} - e^{t_1})(e^{t_2}/\sqrt{2} + e^{t_1}) = \frac{e^{2t_2}}{2} - e^{2t_1} \]
and
\[ m(\mathcal{Z} ( F(t_1,t_2))) = \frac{\pi}{4} (  e^{2t_2} - e^{2t_1}) .\]

It follows that
\[ \frac{  m(\mathcal{Z}( F(t_1,t_2) ) )}{m(T)} = \frac{\pi}{4} \left ( \frac{e^{2(t_2-t_1)} -1}{e^{2(t_t-t_1)}/2 -1} \right ). \]
By the hypotheses, $t_2-t_1\geq a$ and we may assume by Lemma \ref{lem:estimate} that $e^{2a}>3$. The real function $g(t) = \frac{t-1}{t/2-1}$ is decreasing on $(2,\infty)$ and it follows that we obtain
\begin{equation}
    \label{eq:double1}
\frac{  m(\mathcal{Z}( F(t_1,t_2) ) )}{m(T)} \leq \frac{\pi g(3)}{4} < 2\pi.
\end{equation}

Next, Lemma \ref{lem:estimate} implies that the width of the trapezoid $T$ is at least $2\pi$.
Suppose first that $T$ has width between $2\pi$ and $4\pi$ so that 
\begin{equation}
    \label{eq:Tsize}
    m(T) \leq 4\pi(2e^{t_1} + 4\pi).
\end{equation} 
Moreover, as $t_1>\ln | \ln|x_0||$, the condition that $e^{t_2}/\sqrt{2} - e^{t_1} \leq 4\pi$ can be rearranged to
\begin{equation}
    \label{eq:double2}
e^{t_2-t_1} \leq \sqrt{2}\left( 1+\frac{4\pi}{|\ln |x_0||}\right ).
\end{equation}

Then $T$ contains a column made up of squares taken alternately from $G(B_0 \cap H)$ and $G(B_1 \cap H)$. As the left-hand edge of $T$ has length $2e^{t_1}$, it follows that this column has a total of $2\lceil e^{t_1} \rceil$ squares. By Lemma \ref{lem:radius}, each pair of squares will contain a disk contained in $\widetilde{V}$ of radius at least $r_0(8\lambda e^{|x_0|})^{-1}$.

It follows from this and \eqref{eq:Tsize} that 
\begin{equation}
    \label{eq:Tsmall}
\frac{ m(\widetilde{V}\cap T)}{m(T)} \geq \frac{\lceil e^{t_1} \rceil \cdot \pi r_0^2(8\lambda e^{|x_0|})^{-2} }{ 4\pi ( 2e^{t_1} + 4\pi) } \geq \frac{r_0^2}{2^9 \lambda^2 e^{2|x_0|} }.   
\end{equation}
Putting our various estimates together, by Lemma \ref{lem:4}, Proposition \ref{prop:reldist}, \eqref{eq:double1}, \eqref{eq:double2} and \eqref{eq:Tsmall}, we obtain
\begin{align*}
\frac{ m(f^{-1}(U) \cap F(t_1,t_2) ) }{m(F(t_1,t_2))} &\geq \frac{1}{\lambda^2 e^{t_2-t_1}} \cdot \frac{ m(\widetilde{V} \cap \mathcal{Z}(F(t_1,t_2)) }{m(\mathcal{Z}(F(t_1,t_2)))} \\
&\geq \frac{1}{\lambda^2 e^{t_2-t_1}} \cdot \frac {m(\widetilde{V} \cap T)}{ m(T) \cdot 2\pi} \\
&\geq \frac{ r_0^2}{ 2^{11}\lambda^4 \pi e^{2|x_0|}(1 + 4\pi / |\ln |x_0||) }.
\end{align*}

Now, if $T$ has width larger than $4\pi$, then we can split $T$ up into trapezoids which have width between $2\pi$ and $4\pi$ and so the proportion estimate above will persist.
\end{proof}

\subsection{Projections}
\label{ss:proj}

Recall the parametrization of lines through $P$ given by $Y$ in \eqref{eq:para}. As $Y$ has four faces, without loss of generality we may work with the face given by $x_1=1$, denoted by $Y_1$. The argument for the other three faces follows analogously.

For $M>>0$, let $\Omega_M$ be the plane
\[ \Omega_M = \left \{ x_1 = \frac{\pi}{2} + M \pi \right \}.\]
We define $\Pi_M$ to be the nearest point projection of $Y_1$ from $P$ onto $\Omega_M$.

\begin{lemma}
    \label{lem:M}
The projection $\Pi_M$ maps $Y_1$ conformally into $\Omega_M$. In particular, the relative distortion $D(\Pi_M , Y_1) = 1$.
\end{lemma}

\begin{proof}
Consider $u \in Y_1$ so that $u = (p_1+1 , p_2 + u_2 , p_3 + u_3)$ for $u_2\in (-1,1)\setminus \{ 0 \}$ and $u_3 > 0$. Then $\Pi_M(u) = (\pi/2 + M\pi , v_2,v_3)$ for some $v_2,v_3$ in $\R$.

Suppose we set $P_1 = (p_1+1,p_2, p_3)$, $P_2 = (p_1+1,p_2,p_3+u_3)$, $P_1' = (\pi/2+M\pi,p_2,v_3)$, and $P_2' = (\pi/2+M\pi, p_2,v_3)$. Then $PP_1P_2$ and $PP_1'P_2'$ are similar triangles and it follows that if we set $c = \pi/2+M\pi - p_1$, we obtain
\begin{equation} 
\label{eq:M1}
\frac{v_3-p_3}{u_3} = \frac{c}{1}
\end{equation}
and so
\[ v_3 = cu_3 + p_3.\]
Next, if we set $P_3 = (p_1+1,p_2+u_2,p_3)$, $P_4 = (p_1+1,p_2+u_2,p_3+u_3)$, $P_3' = (\pi/2+M\pi, y_2,p_3)$, and $P_4' = (\pi/2+M\pi, y_2,y_3)$, then $PP_3P_4$ and $PP_3'P_4'$ are similar triangles. From \eqref{eq:M1}, we obtain
\[ \frac{ \sqrt{c^2 + (v_2-p_2)^2 } }{\sqrt{1+u_2^2}} = \frac{v_3 - p_3}{u_3} - c.\]
By simplifying this, we obtain
\[ v_2 = p_2 + cu_2.\]
Putting this together, we have shown that
\begin{align*} 
\Pi_M ( p_1+1 , p_2 + u_2 , p_3 + u_3) &= (\pi/2 + M \pi , p_2 + cu_2 , p_3 + cu_3)\\
&= P + c(1,u_2,u_3),
\end{align*}
from which the lemma follows.    
\end{proof}

The point of introducing $\Omega_M$ is that $f^{-1}(\widetilde{R})$ lies on $f^{-1}(R)$. While there are infinitely many components of $f^{-1}(R)$, once the $x_3$-coordinate is large enough, there are components of $f^{-1}(R)$ very close to $\Omega_M$.

To make this idea more precise, suppose that $M\in \N$ is large and for $l \in \mathbb{Z}$, let $g_l \in G$ be such that the corresponding component $S_{g_l}$ of $Z^{-1}(H)$ lies in a beam adjacent to $\Omega_M$ and intersects the plane $x_2 = l\pi$.  
Moreover, for each $l \in \mathbb{Z}$, let $F_l \subset S_{g_l}$ denote the face of $S_{g_l}$ that is adjacent to $\Omega_M$. 

There are clearly gaps between the union of the faces $F_l$ that we need to address. The idea is to exclude certain strips from $\Omega_M$ and then show that any line from $P$ that goes through the remaining region must also pass through the union of the $F_l$'s.

To that end, suppose that $M>>0$ and $s>0$ are given, fix $l\in \Z$ and for $\eta >0$ define the vertical strip
\begin{equation}
    \label{eq:Kl}
K_l := \left \{ x_1 = \frac{\pi}{2} + M\pi \right \} \times \left ( \frac{\pi}{2} + (l-1)\pi + \eta, \frac{\pi}{2} + l\pi - \eta \right ) \times (s,\infty).
\end{equation}
Then $K_l$ is a subset of $\Omega_M$ adjacent to the face $F_l$.

We will be interested in those $K_l$ that lie in the quadrant we are working in. More precisely, given $M$ and $s$, let $l_0 \in \N$ be such that strips $K_l$, for $l=-l_0,\ldots, l_0$ all lie in the region $x_1 > |x_2|$.

\begin{lemma}\label{intersection-lemma}
Let $\delta>0$. Given $\eta>0$, by Lemma \ref{lem:squarebeam}, let $s_0>0$ be sufficiently large that for $x_3>s_0$ and any $x\in F_l$, we have $\operatorname{dist}(x,\Omega_M) < \eta/3$. Then there exists $M\in \N$ such that if $-l_0 \leq l \leq l_0$, any line segment $L$ from $p$ passing through $Y_1$ and $K_l$, as defined in \eqref{eq:Kl} with the same $\eta$ as above, intersects $F_l$.
\end{lemma}

\begin{proof}
Suppose that $s$ is larger than $s_0$. Then the intersection of $F_l$ with the plane $x_3 = s$ is the line segment
\[
\{\pi/2 + M\pi \pm \xi\} \times [\pi/2 + (l-1)\pi + \xi, \pi/2 + l\pi - \xi] \times \{s\},
\]
where $\sin \xi = e^{-s} |\ln |x_0||$ by the computations in Lemma \ref{lem:squarebeam}, and the $\pm$ corresponds to whether the face $F_l$ is in front or behind $\Omega_M$ from the point of view of $P$.

Now, let $L$ be a ray from $P$ passing through $K_l$. Its projection onto the $(x_1,x_2)$-plane makes an angle $\theta$ with  the $x_1$-axis, where $|\theta| \le \pi/4$. Let
\[
A = (\pi/2 + M\pi , a_2, a_3)
\]
be the intersection of $L$ with $K_l$, with $a_2 \in (\pi/2 + (l-1)\pi + \eta, \pi/2 + l\pi - \eta)$ and $a_3 > s$.

Consider the point $B = (b_1, b_2, b_3) \in L$ where the ray's $x_1$-coordinate matches the face, i.e., $|b_1 - (\pi/2 + M\pi)| = \xi$. As $|\tan \theta| \le 1$,
\[
|b_2 - a_2| \le |\tan \theta| \cdot |b_1 - (\pi/2 + M\pi)| \le 1 \cdot \xi = \xi.
\]
Substituting the bounds for $a_2$, the $x_2$-coordinate $b_2$ satisfies
\[
b_2 \in [\pi/2 + (l-1)\pi + \eta - \xi, \, \pi/2 + l\pi - \eta + \xi].
\]
Since $\xi < \eta/3$, we have $\eta - \xi > 2\eta/3 > \xi$. Thus
\[
b_2 \in (\pi/2 + (l-1)\pi + \xi, \, \pi/2 + l\pi - \xi),
\]
and we conclude the ray $L$ must intersect the face $F_l$.
\end{proof}

We next show that projecting from $Y_1$ onto one of the faces $F_l$ is close to projecting onto $\Omega_M$. To make this more precise, for $-l_0 \leq l \leq l_0$, set
\[ J_l = \Pi_{M}^{-1} ( K_l  ) \subset Y_1 .\]
Then $J_l$ is a subset of $Y_1$ for which any ray from $P$ through $J_l$ is guaranteed to intersect $F_l$. It therefore makes sense to define the nearest point projection from $P$ of $J_l$ to $F_l$ via the map
\[ \Pi_{F_l} : J_l \to F_l.\]     

\begin{lemma}\label{face-distortion}
With the set-up above, if $\delta>0$ is sufficiently small and $s,M$ are sufficiently large, then the relative distortion of $\Pi_{F_l}$ satisfies
\[ D(\Pi_{F_l} , J_l) \leq 2.\]
\end{lemma}

\begin{proof}
Fix $x \in J_l$ and set $y = \Pi_{F_l}(x) \in F_l$. Since $F_l$ is a $C^1$ surface, it admits a well-defined tangent plane at every
point. Let $T_l(y)$ denote the tangent plane to $F_l$ at $y$, and let
\[
\Pi_y : J_l \to T_l(y)
\]
denote the projection of $J_l$ from $P$ onto this tangent plane. 

It follows that the linear distortions of the projections onto $F_l$ and $T_l(x)$ coincide at $x$:
\begin{equation}\label{eq:tangent-stretching}
\tilde L_{\Pi_{F_l}}(x) = \tilde L_{\Pi_y}(x),
\qquad
\tilde l_{\Pi_{F_l}}(x) = \tilde l_{\Pi_y}(x).
\end{equation}

For $x_3 = s$ sufficiently large, the tangent plane $T_l(x)$ is nearly vertical, and so in a neighbourhood of $x$, $T_l(x)$ is very close to $\Omega_M$. Since the pointwise Lipschitz constants of the projection depend continuously on $x$, we may choose $s$ and $M$ large enough so that
\begin{equation}\label{eq:tau-estimate}
 \tilde L_{\Pi_y}(x) \leq \sqrt{2} \tilde L_{\Pi_M}(x) \,
\qquad
\tilde l_{\Pi_y}(x) \geq \frac{\tilde l_{\Pi_M}(x)}{\sqrt{2}},
\quad \text{for all } x \in J_l.
\end{equation}

Combining \eqref{eq:tangent-stretching}, \eqref{eq:tau-estimate}, and Lemma \ref{lem:M}, we get
\[
D(\Pi_{F_l}, J_l)
=
\frac{\sup_{x \in J_l} \tilde L_{\Pi_{F_l}}(x)}
     {\inf_{x \in J_l} \tilde l_{\Pi_{F_l}}(x)}
=
\frac{\sup_{x \in J_l} \tilde L_{\Pi_y}(x)}
     {\inf_{x \in J_l} \tilde l_{\Pi_y}(x)}
\le
2 \frac{\sup_{x \in J_l} \tilde L_{\Pi_M}(x) }
     {\inf_{x \in J_l} \tilde l_{\Pi_M}(x) }
= 2 ,     
\]
as required.
\end{proof}

\subsection{Proof of Theorem \ref{thm:2}}

Let $x = (x_1,x_2,x_3) \in Y$ and, without loss of generality, we may assume that $x\in Y_1$ so that $x_1 = 1$. Let $\delta_0>0$ be small enough that the neighbourhood 
\[ E_{\delta_0} = \{ (1,y_2.y_3) : M(x_2-y_2,x_3-y_3) < \delta_0 \}\]
is contained in $Y_1$ and then suppose that $0<\delta <\delta_0$. We define the neighbourhoods $E_{\delta}$ in the analogous way. 

Given $0<\eta<\pi/4$, by Lemma \ref{lem:squarebeam}, choose
\[ s_0 > \ln \frac{ | \ln |x_0|| }{\sin (\eta / 3)} \]
so that for any $x\in S_g$ with $x_3>s_0$, the distance from $x$ to the boundary of its containing beam ($g(B_0)$ or $g(B_1)$) is less than $\eta / 3$. We then choose $M \in \N$ large enough that the image $\Pi_M(E_{\delta})$ is a square in the plane $\Omega_M$ contained in the half-space $\{x_3>s_0 \}$. Moreover, by Lemma \ref{lem:M}, as $E_{\delta}$ has side-length $2\delta$, it follows that the side-length of $\Pi_M(E_{\delta})$ is $2\delta(\pi/2 + M\pi - p_1)$. By choosing $M$ appropriately, we can ensure that the side-length $w$ of $\Pi_M(E_{\delta})$ satisfies
\begin{equation}
\label{eq:width}
     \max \{ 4\pi , a\}  \leq w \leq 10 \max \{ 4\pi,a\}, 
\end{equation}
where $a$ is the constant from Lemma \ref{lem:estimate}.

Following Section \ref{ss:proj}, we construct the vertical strips $K_l \subset \Omega_M$ that are adjacent to the faces $F_l$ of $S_{g_l}$ for $-l_0\leq l \leq l_0$. We then identify the faces $(F_l)_{l=l_1}^{l_2}$ for $-l_0 \leq l_1 <l_2 \leq l_0$ such that the width of the corresponding strips $K_l$ is fully contained in the width of $\Pi_M(E_{\delta})$.

As the side-length of $\Pi_M(E_{\delta})$ is at least $4\pi$, at least of half of the width of the  square is covered by faces of beams which contribute to an appropriate $K_l$.
It then follows by \eqref{eq:Kl} that the proportion of $\Pi_M(E_{\delta})$ consisting of the corresponding $K_l$'s is
\begin{equation}
    \label{eq:Klprop}
\frac{ m( \Pi_M(E_{\delta}) \cap \bigcup_{l=l_1}^{l_2} K_l ) }{m(\Pi_M(E_{\delta}) )} 
\geq \frac{1}{2} \cdot \frac{ \pi-2\eta}{\pi} \geq \frac{1}{4},
\end{equation}
assuming $\eta$ was initially chosen smaller than $\pi/4$.

Again from Section \ref{ss:proj}, we recall the construction of the subsets $J_l$ of $Y_1$. Set $J_{l,\delta} = J_l \cap E_{\delta}$ for $l_1\leq l \leq l_2$ and $\delta>0$. Then by applying Proposition \ref{prop:reldist} and using $D(\Pi_M,Y_1) = 1$, from \eqref{eq:Klprop} we obtain
\begin{equation}
    \label{eq:Jlprop}
\frac{ m(E_{\delta} \cap \bigcup_{l=l_1}^{l_2} J_{l,\delta}  )}{ m(E_{\delta}) }  = \frac{ m( \Pi_M(E_{\delta}) \cap \bigcup_{l=l_1}^{l_2} K_l ) }{m(\Pi_M(E_{\delta}) )} \geq \frac{1}{4}.
\end{equation}

Our next task is to determine a lower bound for the proportion of $X_U$ contained in each $J_{l,\delta}$. To that end, if the square $\Pi_M(E_{\delta})$ lies between $x_3 = t_1$ and $x_3 = t_2$, following the notation from \eqref{eq:Ftt}, we define $F_l(t_1,t_2)$ to be $F_l \cap \{ t_1<x_3<t_2 \}$. By \eqref{eq:width} and Lemma \ref{lem:doublepreims}, there exists a constant $C$ that depends only on our Zorich map $\mathcal{Z}$ and the choice of the ball $U$ so that
\[ \frac{ m(f^{-1}(U) \cap F_l(t_1,t_2) ) }{m(F_l(t_1,t_2))} \geq C. \]

Applying (the comment after the proof of) Proposition \ref{prop:reldist} with $\Pi_{F_l}$ and using Lemma \ref{face-distortion}, we obtain
\begin{equation}
\label{eq:ineq1}
    \frac{ m(X_U \cap J_{l,\delta})}{m(J_{l,\delta})} \geq \frac{ 1}{D(\Pi_{F_l} , J_{l,\delta})^2} \cdot \frac{ m(f^{-1}(U) \cap F_l(t_1,t_2) ) }{m(F_l(t_1,t_2))} \geq \frac{C}{4}.
\end{equation}

Here we have used the fact that $X_U \cap J_{l,\delta}$ contains $\Pi_{F_l}^{-1} (\mathcal{Z}^{-1}(\widetilde{V})\cap F_l(t_1,t_2)) $. 

Finally, we can obtain our desired estimate. By using \eqref{eq:Jlprop} and \eqref{eq:ineq1}, we obtain
\begin{align*}
m(X_U \cap E_{\delta}) &\geq \sum_{l=1_1}^{l_2} m(X_U \cap J_{l,\delta}) \\
&\geq \frac{C}{4} \sum_{l=l_1}^{l_2}  m(J_l) \\
&\geq \frac{C}{16 } m(E_{\delta}).
\end{align*}

The constant $C/16$ on the right-hand side does not depend on $\delta$, as long as $\delta$ is chosen sufficiently small. Thus, by setting $\epsilon = C/16$, we conclude that $x \in Y_1$ is an $\epsilon$-density point of $X_U$.

This completes the proof of Theorem \ref{thm:2}, and hence also of Theorem \ref{thm:1}.

\end{document}